\documentclass[11p,reqno]{amsart}
\usepackage{amsmath}
\usepackage{amsfonts}
\usepackage{amssymb}
\usepackage{graphicx}
\usepackage{color}
\newtheorem{theorem}{Theorem}[section]
\newtheorem*{theorem*}{Theorem}
\newtheorem{lemma}{Lemma}[section]
\newtheorem{corollary}[theorem]{Corollary}

\newtheorem{example}[theorem]{Example}

\newtheorem{conjecture}[theorem]{Conjecture}

\def\qed{\hfill $\square$}

\newcommand{\eps}{{\varepsilon}}
\numberwithin{equation}{section}

\begin{document}
	\title[positive orthogonal Ricci curvature]{Manifolds with positive orthogonal Ricci curvature}

\author{Lei Ni}\thanks{The research of LN is partially supported by NSF grant DMS-1401500.  }
\address{Lei Ni. Department of Mathematics, University of California, San Diego, La Jolla, CA 92093, USA}
\email{lni@math.ucsd.edu}

\author{Qingsong Wang}
\address{Qingsong Wang. Department of Mathematics,
The Ohio State University, Columbus, OH 43210,
USA}
\email{{wang.8973@osu.edu}}

\author{Fangyang Zheng} \thanks{The research of FZ is partially supported by a Simons Collaboration Grant 355557.}
\address{Fangyang Zheng. Department of Mathematics,
The Ohio State University, Columbus, OH 43210,
USA}
\email{{zheng.31@osu.edu}}

%\subjclass[2010]{32L05, 32Q10, 32Q15, 53C55}
%\keywords{Compact complex manifolds, K\"ahler metrics, positive holomorphic sectional curvature, positive scalar
%\ curvature, projectivized vector bundles}

%\date{}
\begin{abstract} In this paper we  study the class of compact K\"ahler manifolds with $Ric^\perp>0$. First we illustrate  examples of K\"ahler manifolds with $Ric^\perp>0$ on K\"ahler C-spaces, and construct ones on certain projectivized vector bundles. These examples show the abundance of K\"ahler manifolds which admit metrics of $Ric^\perp>0$. Secondly  we  prove some (algebraic) geometric consequences of the condition $Ric^\perp>0$ to illustrate that the condition  is also quite restrictive. Finally this last point is made evident with a classification result in dimension three and a partial classification in dimension four.
\end{abstract}

\maketitle

\section{Introduction}

In a recent work \cite{NZ} by the first and third author, the geometric implications  of {\em orthogonal Ricci curvature} $Ric^{\perp}$ on a  K\"ahler manifold $M^n$, which is defined by
$$ Ric^{\perp}_{X\overline{X}} = Ric(X, \overline{X}) - R(X, \overline{X}, X, \overline{X})/|X|^2 $$
for any type $(1,0)$ tangent vector $X$, was studied. For a compact K\"ahler manifold  with $Ric^{\perp}>0$ everywhere, it was shown in \cite{NZ} that the manifold is always projective, has finite $\pi_1(M)$, and has vanishing Hodge numbers: $h^{p,0}=0$ for $p=1, \,2, \,n-1$, and $n$. Beside the results just mentioned and the comparison theorems relating $Ric^\perp$ to $\Delta^\perp$, some compact and noncompact examples of K\"ahler manifolds with $Ric^\perp>0$ were also illustrated there.

The goal of this paper is to continue the study of compact K\"ahler manifolds with positive orthogonal Ricci curvature. Let us denote by ${\mathcal M}_n^{\perp}$ the set of all compact complex manifolds of complex dimension $n$ which admit K\"ahler metrics with $Ric^{\perp }>0$. For each $n\geq 2$, the ultimate goal is to  understand the class ${\mathcal M}_n^{\perp}$. In this paper we shall first illustrate  examples of manifolds with $Ric^\perp>0$ by showing that the K\"ahler-Einstein metrics  on most K\"ahler C-spaces with $b_2=1$ satisfy $Ric^\perp>0$, and then construct such K\"ahler metrics  on some projectivized vector bundles.

 In the second part of the paper we will prove some (algebraic) geometric consequences of the condition $Ric^\perp > 0$ to illustrate that the condition is also quite restrictive. Finally in dimension three  a  classification result is obtained, while in dimension four we obtain a partial classification result.  It is our hope that this paper will help to establish  that ${\mathcal M}_n^{\perp}$ forms an interesting class of rationally connected manifolds, perhaps similar to the class of Fano manifolds, and worth further investigation from both the differential geometric and the algebraic geometric point of view.

On the existence side, first of all, it is clear that the $Ric^{\perp}>0$ condition requires that the dimension of the manifold to be at least $2$ since all Riemann surfaces have $Ric^\perp\equiv0$. (In higher dimensions, as we shall see in the Appendix, any K\"ahler manifold $M^n$ with $n\geq 2$ and $Ric^\perp =0$ everywhere is either flat, or has $n=2$ and is locally holomorphically isometric to the product of two complex curves, with constant curvature of opposite signs.)  For a product manifold with the product metric, the product will have $Ric^{\perp}>0$ if  both factors have $Ric^{\perp}>0$ and $Ric \geq 0$. So for $X^n\in {\mathcal M}_n^{\perp}$ and $Y^m\in {\mathcal M}_m^{\perp}$, if the metrics involved also have nonnegative Ricci, then the product manifold $X\times Y$ lies in ${\mathcal M}_{n+m}^{\perp}$.   Observe also, that any small deformation of a manifold in ${\mathcal M}_n^{\perp}$ is again in ${\mathcal M}_n^{\perp}$.

Built upon the works by Itoh \cite{Itoh} as well as  Chau and Tam in \cite{ChauTam}, we illustrate that most of the K\"ahler C-spaces with $b_2=1$ admit K\"ahler metrics with $Ric^\perp>0$.

\begin{theorem}
Let $M^n$ be a classical K\"ahler C-space with $n\geq 2$ and $b_2=1$. Then the (unique up to constant multiple) K\"ahler-Einstein metric has $Ric^{\perp}>0$.
\end{theorem}

K\"ahler C-spaces with $b_2=1$ consists of four {\it classical} sequences, plus finitely many exceptional ones. We believe that all such spaces with $n\geq 2$ (namely, except ${\mathbb P}^1$) have $Ric^{\perp}>0$. However, since our computation is based upon the curvature computations in \cite{Itoh}, which was done only for the classical ones, plus some but not all of the exceptional ones, we can not claim the result for all exceptional cases before carrying out the computation of holomorphic sectional curvature for all of them. Notice that the above result, together with Chau-Tam's work on the quadratic bisectional curvature, provides {\it many compact homogenous examples with $Ric^\perp>0$, but with negative  quadratic bisectional curvature  somewhere}.

For K\"ahler C-spaces with $b_2>1$, the conclusion of Theorem 1.1 is no longer true in general. For instance, as we shall see in discussion a bit later, the flag threefold ${\mathbb P}(T_{ {\mathbb P} \,^2} ) $, which is a K\"ahler C-space with $b_2=2$, can not admit any K\"ahler metric with $Ric^{\perp}>0$. It would be an interesting question to know what kind of K\"ahler C-spaces are in ${\mathcal M}_n^{\perp}$.

Since there are only two irreducible compact Hermitian symmetric spaces that are exceptional, in view of  \cite{Itoh},  we have the  following result.

\begin{corollary}
Let $M^n$ be a compact Hermitian symmetric space without any ${\mathbb P}^1$ factor. Then it has $Ric^{\perp}>0$.
\end{corollary}

On the other hand, as we shall see later, for any compact complex manifold $N$, ${\mathbb P}^1\times N$ can never admit a K\"ahler metric with $Ric^{\perp}>0$.

 Another set of examples can be obtained by considering projectivized vector bundles. Let $(M^n,g)$ be a compact K\"ahler manifold and $(E,h)$ be a holomorphic vector bundle of rank $r$ over $M^n$, equipped with a Hermitian metric $h$. Let $\pi : P={\mathbb P}(E^{\ast}) \rightarrow M$ be the projectivization of $E$, namely, for $x\in M$, $\pi^{-1}(x) = {\mathbb P}(E_x)$ is the set of all complex lines through origin in $E_x$. Note that it is the tradition in algebraic geometry to denote this space as ${\mathbb P}(E^{\ast})$. Consider the K\"ahler metric $G$ on $P$ with K\"ahler form
$$ \omega_G = \lambda\, \pi^{\ast }\omega_g + C_1(L, \hat{h}), $$
where $\lambda >0$ is a sufficiently large constant, $L$ is the dual of the tautological line bundle on $P$, and $\hat{h}$ is the metric on $L$ induced by $h$. At a point $(x,[v])\in P$, where $x\in M$ and $0\neq v\in E_x$, $C_1(L,\hat{h})$ is given by
$$ C_1(L,\hat{h}) = \omega_{\mbox{\tiny{FS}}} - \frac{\sqrt{\!-\!1}}{\ |v|^2} \Theta^h_{v\overline{v}} $$
where $\Theta^h$ is the curvature form of $(E,h)$ and $\omega_{\mbox{\tiny{FS}}}$ is the K\"ahler form of the Fubini-Study metric on the fiber of $\pi$. We have the following:

\begin{theorem}
Let $(M^n,g)$ be a compact K\"ahler manifold with $Ric^{\perp}>0$, and $(E,h)$ be a Hermitian vector bundle over $M$ of rank $r\geq 3$ such that for any $x\in M$ and any $0\neq v\in E_x$,
\begin{equation}
 Ric^{g\perp}_{X\overline{X}} + R(\det E)_{X\overline{X}} - \frac{r}{|v|^2}R^h_{v\overline{v} X\overline{X}   } \ > \ 0   \label{eq:curvature}
 \end{equation}
for any tangent vector $0\neq X\in T^{1,0}_xM$. Here $R(\det E)$ is the curvature of the determinant line bundle $\det E = \bigwedge^{\!r}\!E$ equipped with the metric induced by $h$. Then on the projectivized bundle $P={\mathbb P}(E^{\ast})$, the K\"ahler metric $G$ with $\omega_G=\lambda \,\pi^{\ast}\omega_g + C_1(L,\hat{h})$ will have $Ric^{\perp }>0$ everywhere when $\lambda$ is sufficiently large.
\end{theorem}

Note that the rank requirement $r\geq 3$ here is necessary, as we shall see later that any ${\mathbb P}^1$-bundle over any base space can never admit a K\"ahler metric with $Ric^{\perp}>0$.

The curvature condition (\ref{eq:curvature}) is independent of the scaling of metrics $g$ or $h$, as well as tensoring of $E$ by a line bundle. When the dimension of the base manifold is $3$ or higher,  the above theorem gives non-trivial examples of manifolds with $Ric^{\perp}>0$. For instance, when $M={\mathbb P}^n$ ($n\geq 3$) equipped with the standard metric, then for the cotangent bundle $\Omega_{   {\mathbb P}^{n}  }  $ or $E={\mathcal O}^{\oplus 2} \oplus {\mathcal O}(-1)$, with the standard metrics, one can easily check that the condition (\ref{eq:curvature}) is satisfied, hence we have the following examples:

\begin{example}\label{ex:14}
For any $n\geq 3$, the $(2n-1)$-dimensional manifold ${\mathbb P}(T_{{ \mathbb P}^n})$ and the $(n+2)$-dimensional manifold ${\mathbb P}({\mathcal O}_{{ \mathbb P}^n}^{\oplus 2} \oplus {\mathcal O}_{{ \mathbb P}^n}(1) )$ are in ${\mathcal M}^{\perp}$. Similarly, consider the splitting bundle $E={\mathcal O}(a_1) \oplus \cdots \oplus {\mathcal O}(a_r)$ over ${\mathbb P}^n$, where $a_1\geq a_2 \geq \cdots \geq a_r$. If
$$ n-1 > (a_1-a_2) + \cdots + (a_1-a_r),$$
then ${\mathbb P}(E^{\ast })$ will be in ${\mathcal M}^{\perp}$.
\end{example}

In particular, the Fano fivefold ${\mathbb P}({\mathcal O}^{\oplus 2} \oplus {\mathcal O}(1))$ over ${\mathbb P}^3$, which is not a complex homogeneous space as it has a section with negative normal bundle, admits K\"ahler metrics with $Ric^{\perp}>0$.

As an interesting contrast, when the base is two dimensional, any ${\mathbb P}^k$-bundle cannot be in ${\mathcal M}^\perp$ unless it is the product:

\begin{theorem}\label{thm:15}
Let $k\geq 2$, and let $P$ be a holomorphic fiber bundle over a compact complex surface $S$, whose fiber is ${\mathbb P}^k$. If $P$ admits a K\"ahler metric with  $Ric^{\perp}>0$, then  $S$ must be biholomorphic to ${\mathbb P}^2$ and $P$ must be biholomorphic to ${\mathbb P}^2\times {\mathbb P}^k$.
\end{theorem}

 So in particular, any non-trivial ${\mathbb P}^k$-bundle over ${\mathbb P}^2$ does not admit any K\"ahler metric with $Ric^{\perp}>0$, even though both the fiber and the base do.

 On the non-existence side,  an observation to the condition $Ric^\perp>0$ is the  following result which generalizes a theorem of Frankel \cite{Frankel}:

\begin{theorem} \label{thm:16}
Let $M^n$ be a compact K\"ahler manifold with $Ric^{\perp}>0$. If $Y_1$ and $Y_2$  are irreducible divisors in $M$, then $Y_1\cap Y_2 \neq \phi$.
\end{theorem}

As an immediate corollary, we know that  manifolds with $Ric^\perp$ cannot be the blowing up of a (smooth or singular) point, or a fiberation over a curve:

\begin{corollary}
Let $M^n$ be a compact K\"ahler manifold with $Ric^{\perp}>0$. Then there exists no surjective holomorphic map from $M^n$ onto a complex curve, and there exists no birational morphism $f: M \rightarrow Z$ onto a normal variety $Z$, where a smooth hypersurface in $M$ is mapped to a (smooth or singular) point.
\end{corollary}

 A Lefschetz type theorem can also be proved for compact K\"ahler manifolds with $Ric^\perp>0$. Namely {\it for any smooth complex hypersurface $Y$, the induced map $\iota_*: \pi_1(Y)\to \pi_1(M)$ is surjective.}

For $n=2$, $Ric^{\perp}$ is the same as orthogonal bisectional curvature. So the result of \cite{GuZhang} implies that the only $M^2$ which admits a K\"ahler metric with $Ric^{\perp}>0$ is  ${\mathbb P}^2$. It turns out that in dimension $3$ and $4$, $\mathcal{M}^\perp_3$ and $\mathcal{M}^\perp_4$ are also rather small, thanks to the powerful cone-contraction theorems by Mori \cite{Mori} and Koll\'ar \cite{Kollar} and the numerous follow up works afterwards. In dimension three we have the following

\begin{theorem}\label{thm:18}
Let $M^3$ be a compact K\"ahler manifold with $Ric^{\perp}>0$, then $M^3$ is biholomorphic to either ${\mathbb P}^3$  or ${\mathbb Q}^3$, the smooth quadratic hypersurface in ${\mathbb P}^4$.
\end{theorem}

In dimension four, we only have a partial result:

\begin{theorem}\label{thm:19}
Let $M^4$ be a compact K\"ahler manifold with $Ric^{\perp}>0$, then $M^4$ is biholomorphic to either ${\mathbb P}^2\times {\mathbb P}^2$,   or a Fano fourfold with $b_2=1$ and with pseudo index $i\geq 3$.
\end{theorem}

The pseudo index $i(M)$ of a Fano manifold $M^n$ is defined to be the minimum of the intersection number $K^{\!-1\!}_{\!M} \,C$,  where $K^{\!-1\!}_{\!M}$ is the anti-canonical line bundle and $C$ is any rational curve in $M$.

Recall that a {\em del Pezzo manifold} $M^n$ is defined as a Fano manifold with index $n-1$, where the index is the largest integer $r$ such that $K_M^{-1}=rA$ for an ample divisor $A$. For $n\geq 3$, such manifolds were completely classified by Fujita in \cite{Fujita}, arranged by their degree $d$ which is defined as $A^n$. They are:

$\bullet$ $d=1$: $X^n_6\subset {\mathbb P}(1^{n-1}, 2, 3)$, a degree $6$ hypersurface in the weighted projective space.

$\bullet$ $d=2$: $X^n_4\subset {\mathbb P}(1^{n}, 2)$, a degree $4$ hypersurface in the weighted projective space.

$\bullet$ $d=3$: $X^n_3\subset {\mathbb P}^{n+1}$, a cubic hypersurface.

$\bullet$ $d=4$: $X^n_{2,2}\subset {\mathbb P}^{n+2}$, a complete intersection of two quadrics.

$\bullet$ $d=5$: $Y^n$, a linear section of ${\mathbb Gr}(2,5) \subset {\mathbb P}^9$.

$\bullet$ $d=6$: ${\mathbb P}^1\!\times \! {\mathbb P}^1\!\times \!{\mathbb P}^1$, or ${\mathbb P}^2\!\times \! {\mathbb P}^2$, or the flag threefold ${\mathbb P}(T_{ \!{\mathbb P}^{2}} )$.

$\bullet$ $d=7$: ${\mathbb P}^3\# \overline{{\mathbb P}^3} $, the blowing up of ${\mathbb P}^3$ at a point.

We propose the following

\begin{conjecture}
A compact complex manifold $M^4$ of dimension $4$ admits a K\"ahler metric with $Ric^{\perp}>0$ if and only if $M^4$  is biholomorphic to  ${\mathbb P}^4$, or $ {\mathbb Q}^4 $, or a del Pezzo fourfold: $X^4_6$, $X^4_4$, $X^4_3$, $X^4_{2,2}$, $Y^4$, or
${\mathbb P}^2\!\times \!{\mathbb P}^2$.
\end{conjecture}

In other words, we conjecture that for $n\leq 4$, the set of all compact K\"ahler $n$-manifolds with $Ric^\perp>0$ coincide with the set of all Fano $n$-folds with index $r\geq 3$.

For $n\geq 5$, the set ${\mathcal M}^{\perp}_n$ contains more examples, and the index could certainly be $1$, e.g., ${\mathbb P}^2\times {\mathbb P}^3$, or ${\mathbb P}({\mathcal O}^{\oplus 3}\oplus {\mathcal O}(1))$ over ${\mathbb P}^3$. We believe that all manifolds in ${\mathcal M}^{\perp}_n$ should be rationally connected. We do not know whether or not they should all be Fano, even though all examples constructed so far are Fano. One could even ask if all such manifolds admit K\"ahler-Einstein metrics. In any event, for $n\geq 5$, ${\mathcal M}^{\perp}_n$ should form a very interesting class of algebraic manifolds, which perhaps worths some attention from both differential geometers and algebraic geometers.

\vspace{0.5cm}

\section{K\"ahler C-spaces}

First let us recall the well known fact about K\"ahler C-spaces, they are exactly the orbit spaces of the adjoint representation of compact simple Lie groups. We will follow the discussion of \cite{Itoh} or \cite{ChauTam} and references therein. All K\"ahler C-spaces with $b_2=1$ can be described as follows. Let ${\mathfrak g}$ be a simple complex Lie algebra. They are fully classified as the four classical sequences $A_r={\mathfrak s}{\mathfrak l}_{\,r+1}$ ($r\geq 1$), $B_r={\mathfrak s}{\mathfrak o}_{2r+1}$ ($r\geq 2$), $C_r={\mathfrak s}{\mathfrak p}_{2r}$ ($r\geq 3$), $D_r={\mathfrak s}{\mathfrak o}_{2r}$ ($(r\geq 4$) and the exceptional ones $E_6$, $E_7$, $E_8$, $F_4$ and $G_2$.

Let ${\mathfrak h}\subset {\mathfrak g}$ be its Cartan subalgebra with corresponding root system $\Delta \subset {\mathfrak h}^{\ast}$, so we have ${\mathfrak g} = {\mathfrak h} \oplus \bigoplus_{\alpha \in \Delta} {\mathbb C}E_{\alpha }$ where $E_{\alpha}$ is a root vector of $\alpha$. Let $r=\dim_{\mathbb C} {\mathfrak h}$ and fix a fundamental root system $\{ \alpha_1, \ldots , \alpha_r\}$. This gives an ordering in $\Delta$, and let $\Delta^+$, $\Delta^-$ be the set of positive or negative roots. Fix an integer $i$ with $1\leq i\leq r$. For any positive integer $k$, denote by
$$ \Delta_i^+\!(k)= \{ \alpha = \sum_{j=1}^r n_j\alpha_j \in \Delta^+ \mid n_i=k\} $$
and write $\Delta^+_i=\bigcup_{k>0} \Delta^+_i\!(k)$. Let $G$ be the simply connected simple complex Lie group with Lie algebra ${\mathfrak g}$, and $P\subset G$ the parabolic subgroup whose Lie algebra is ${\mathfrak h} \oplus \bigoplus_{\alpha \,\in \Delta \setminus \Delta_i^+} {\mathbb C}E_{\alpha }$. Then $M=G/P$ is a K\"ahler C-space with $b_2=1$. Conversely, any K\"ahler C-space with $b_2=1$ are obtained this way. Following \cite{Itoh}, we will denote this K\"ahler C-space by $({\mathfrak g}, \alpha_i)$.

Let $\{ E_{\alpha}\}_{\alpha \in \Delta} \cup \{ H_{\alpha_j}\}_{j=1}^r$ be a {\em Weyl canonical basis}  of ${\mathfrak g}$ (see \cite{ChauTam} and the references therein), and  write ${\mathfrak m}^{\!+\!k} = \bigoplus_{\alpha \in \Delta_i^{\!+\!}(\!k)} {\mathbb C} E_{\alpha}$, ${\mathfrak m}^{\!-\!k} = \bigoplus_{\alpha \in \Delta_i^{\!+\!}(\!k) }{\mathbb C}E_{-\alpha}$. Then ${\mathfrak m}^{\!+} = \bigoplus_{k>0}{\mathfrak m}^{\!+\!k}$ is the holomorphic tangent space of $M$ at the base point, and the  metric $g$ on $({\mathfrak g}, \alpha_i)$ given by
$$ g =  \sum_{k>0} (-kB)|_{{\mathfrak m}^{\!+\!k} \times {\mathfrak m}^{\!-\!k}} $$
is the unique (up to constant multiple) K\"ahler-Einstein metric on $M$. Here $B$ is the Killing form on ${\mathfrak g}$. Let $e_{\alpha } = \frac{1} {\sqrt{k}} E_{\alpha} $ for $\alpha \in \Delta^{\!+\!}_i(k)$, then $\{  e_{\alpha } \}_{\alpha \in \Delta^{\!+\!}_i} $ forms a unitary (left invariant) frame on $M$, called the {\em Weyl frame}.

Note that K\"ahler C-spaces with $b_2=1$ include all the irreducible compact Hermitian symmetric spaces:
\begin{eqnarray*}
&& (A_r, \alpha_i) =  Gr_{\mathbb C}(i,r\!+\!1),   \ \ \  \ 1\leq i\leq r, \ r\geq 1\,;\\
&& (B_r, \alpha_1) =  {\mathbb Q}^{2r-1} , \ \  \ \  r\geq 2 \,;   \\
&& (B_r, \alpha_r) = \left\{ \begin{array}{ll} {\mathbb P}^3, \ \  \ \ \ \ \mbox{if} \ r=2\,; \\ I\!I_{r+1},  \ \ \mbox{if} \  r\geq 3\,; \end{array} \right. \\
&& (C_r,\alpha_1) = {\mathbb P}^{2r-1} , \ \ \ \ r\geq 3;   \\
&& (C_r, \alpha_r) = I\!I\!I_r, \ \ \ \ r\geq 3\,; \\
&&  (D_r, \alpha_1) = I\!I_{r+1}, \ \ \ r\geq 4;    \\
&&  (D_r, \alpha_{r\!-\!1}) = (D_r, \alpha_r) = I\!I_r, \ \ \ \ r\geq 4\,;\\
&& (E_6, \alpha_1) = (E_6,\alpha_6) = M_{\! \mbox{\it \tiny{V}}   }^{16}\,;    \\
&& (E_7, \alpha_7) = M_{\! \mbox{\it \tiny{V\!I}}   }^{27}\,;\\
&& (G_2, \alpha_1) = {\mathbb Q}^5\,,
\end{eqnarray*}
where $I\!I_n =SO(2n)/U(n)$ is the space of orthogonal complex structures on ${\mathbb R}^{2n}$ and $I\!I\!I_n=Sp(n)/U(n)$ is the space of complex structures on ${\mathbb H}^n$ compatible with the inner product, with ${\mathbb H}$ the quaternions. The former has complex dimension $\frac{1}{2}n(n-1)$ and rank $[\frac{n}{2}]$, and the latter has dimension $\frac{1}{2}n(n+1)$ and rank $n$.

The set of K\"ahler C-spaces with $b_2=1$ which are not Hermitian symmetric spaces consists of the classical sequences
$$ (B_r,\alpha_i)_{r\geq 3}, \ \ (C_r,\alpha_i)_{r\geq 3}, \  \ (D_r,\alpha_i)_{r\geq 4}, $$
where $1<i<r$ for $B_r$ and $C_r$ and $1<i<r-1$ for $D_r$, and  the exceptional ones:
$$ (E_6,\alpha_i)_{2\leq i\leq 5}, \ \ (E_7, \alpha_i)_{1\leq i\leq 6}, \ \ (E_8, \alpha_i)_{1\leq i\leq 8}, \ \ (F_4, \alpha_i)_{1\leq i\leq 4}, \ \ (G_2, \alpha_2).$$

For a simply connected irreducible compact K\"ahler manifold $(M^n,g)$, if the bisectional curvature (or orthogonal bisectional curvature) is non-negative, then by the work of Mok \cite{Mok} (or the work of Gu and Zhang \cite{GuZhang}), either $M^n$ is biholomorphic to ${\mathbb P}^n$ or $(M^n,g)$ is holomorphically isometric to a compact Hermitian symmetric space of rank at least $2$. In \cite{WuYauZheng}, a weaker curvature condition was considered: a K\"ahler manifold $(M^n,g)$ is said to have  nonnegative {\em quadratic bisectional curvature,} denoted by $QB\geq 0$, if at any $x\in M$, for any unitary frame $\{e_1, \ldots , e_n\}$ at $x$, and for any real constants $\{ a_1, \ldots , a_n\} $, it holds that
$$ \sum_{i,j=1}^n R_{i\overline{i} j\overline{j}} (a_i-a_j)^2 \geq 0 .$$
$M$ is said to have positive quadratic bisectional curvature, denoted by $QB>0$, if the above quantity is positive whenever these $a_i$ are not all equal.  The quantity appeared first in \cite{BishopGoldberg} from the Bochner formula in computing the Laplacian of the length square of a $(1, 1)$ form.

It was hoped then that the condition $QB\geq 0$ would be satisfied by all K\"ahler C-spaces with $b_2=1$. In \cite{LiWuZheng}, this was verified for the $7$-dimensional space $(B_3, \alpha_2)$, using the computation of \cite{Itoh}. However, it turns out that the condition $QB\geq 0$ was only satisfied by about 80\% of K\"ahler C-spaces with $b_2=1$, namely, in \cite{ChauTam}, Chau and Tam completely computed $QB$ for all K\"ahler C-spaces with $b_2=1$ excluding the Hermitian symmetric ones, and their conclusions are the following:

\begin{theorem}[Chau-Tam]
For $1<i<r$ and $r\geq 3$, $(B_r, \alpha_i)$ has $QB>0$ ($\geq 0$) if and only if $5i+1< 4r$ ($\leq 4r$). \\
For $1<i<r$ and $r\geq 3$, $(C_r, \alpha_i)$ has $QB>0$ ($\geq 0$) if and only if $5i-3< 4r$ ($\leq 4r$). \\
For $1<i<r\!-\!1$ and $r\geq 4$, $(D_r, \alpha_i)$ has $QB>0$ ($\geq 0$) if and only if $5i+3< 4r$ ($\leq 4r$). \\
For the exceptional ones, the following satisfy $QB>0$:
$$(G_2, \alpha_2), \ \ (F_4, \alpha_i)_{i=1,2,4}, \ \ (E_6,\alpha_i)_{i=2,3,5},  \ \ (E_7,\alpha_i)_{i=1,2,5}, \ \ (E_8,\alpha_i)_{i=1,2,8}.$$
For the remaining ones,  each of them does not satisfy $QB\geq 0$:
$$ {\mathcal E}_0 = \{ (F_4, \alpha_3), \ (E_6, \alpha_4), \ (E_7,\alpha_i)_{i=3,4,6}, \ (E_8,\alpha_i)_{i=3,4,5,6,7}\}. $$
\end{theorem}

Clearly, if we take all but one of those $a_i$ to be zero in the definition of $QB$, we see that the condition $QB>0$ (or $\geq 0$) implies $Ric^{\perp}>0$ (or $\geq 0$). So by the Theorem of Chau and Tam, we know that at least 80\% of K\"ahler C-spaces with $b_2=1$ will satisfy $Ric^{\perp}>0$. It is probably true that all of them except ${\mathbb P}^1$ satisfy $Ric^{\perp}>0$, which involves the verification of $H< \mu$, where $H$ is the holomorphic sectional curvature of any tangent direction, and $\mu$ is the (constant) Ricci curvature of $M$.

We will take advantage of the calculations of \cite{ChauTam} and \cite{Itoh} to conclude Theorem 1.1 and Corollary 1.2, namely, all K\"ahler C-spaces with $b_2=1$ except ${\mathbb P}^1$ or those in ${\mathcal E}_0$ satisfy $Ric^{\perp}>0$, and we believe that those in ${\mathcal E}_0$ will also satisfy $Ric^{\perp}>0$ but we skip its verification to avoid digression from our main line of discussions.

Let us start with the verification of $Ric^{\perp}>0$, or equivalently $H<\mu$, for any irreducible compact Hermitian symmetric spaces $M^n$ other than ${\mathbb P}^1$. First for ${\mathbb P}^n$ with $n\geq 2$, in this case $H$ is constantly  $2$, and $\mu = n+1$, so $H<\mu$ holds. For the quadric hypersurface ${\mathbb Q}^n$, with $n\geq 3$ (note that ${\mathbb Q}^2={\mathbb P}^1\times {\mathbb P}^1$ is reducible), it can be holomorphically and isometrically embedded in ${\mathbb P}^{n+1}$, so its maximum $H$ is again no greater than $2$, while its $\mu$ is $n+2-2=n$, so again we have $H<\mu$. For the complex Grassmann manifold $M^n= Gr_{\mathbb C}(i,r+1)=(A_r, \alpha_i)$, where $1\leq i\leq r$, we have $n=i(r+1-i)$ and $r\geq 2$ (otherwise $M={\mathbb P}^1$). As is well known, $M^n$ can be holomorphically and isometrically embedded in ${\mathbb P}^N= {\mathbb P} (\bigwedge^{\!i} \!{\mathbb C}^{r\!+\!1})$, and the Ricci curvature of $M^n$ is $\mu = r+1$, while the maximum $H$ of $M^n$ is no greater than that of ${\mathbb P}^N$ which is $2$, so $H\leq 2< \mu$, and $M^n$ satisfies $Ric^{\perp}>0$.

For type $I\!I$ and type $I\!I\!I$ Hermitian symmetric spaces, which are $M^n_{\mbox{\it \tiny{I\!I}}}= (B_r, \alpha_r)$ and $M^n_{\mbox{\it \tiny{I\!I\!I}}}= (C_r, \alpha_r)$, respectively, where $r\geq 3$ and $n=\frac{1}{2}r(r+1)$, we will postpone the verification of $Ric^{\perp}>0$ and do it with the other classical K\"ahler C-spaces with $b_2=1$.

For the two exceptional Hermitian symmetric spaces, $M^{16}_{\mbox{\it \tiny{V}}}= (E_6, \alpha_1)$ and $M^{27}_{\mbox{\it \tiny{V\!I}}}= (E_7, \alpha_7)$, again by the computation in the $E_6$ and $E_7$ subsections of \cite{ChauTam}, we see that $\Delta^{\!+\!}_1(k)=\phi$ for any $k\geq 2$ and $\mu=12$  in the former case, while  $\Delta^{\!+\!}_7(k)=\phi$ for any $k\geq 2$  and $\mu=18$ in the latter case. On the other hand, let us recall the so-called {\em curvature operator} $Q$ defined in \cite{Itoh} (note that this is not the curvature operator in Riemannian or K\"ahler geometry). Consider the symmetric product space $S^2T^{1,0}_M$ equipped with the Hermitian inner product
$$ (X\!\cdot \!Y, \overline{Z\!\cdot \!W}) = \frac{1}{2}\big( \langle X,\overline{Z}\rangle  \langle Y,\overline{W}\rangle +   \langle X,\overline{W}\rangle  \langle Y,\overline{Z}\rangle \big) ,$$
where $X\!\cdot Y= \frac{1}{2}(X\otimes Y+Y\otimes X)$. Now consider  $Q: S^2T^{1,0}_M \rightarrow S^2T^{1,0}_M$, the self-adjoint linear operator  defined by
$$ ( Q(X\!\cdot \!Y), \, \overline{Z\!\cdot \!W}) = R_{X\overline{Z}Y\overline{W}}. $$
Denote by $\nu$ the largest eigenvalue of $Q$. Since  $R_{X\overline{X}X\overline{X}} = (Q(X\!\cdot \!X), \, \overline{X\!\cdot \!X})$, the maximum of holomorphic sectional curvature $H$ is no greater than $\nu$. The number $\nu$ was computed in \cite{Itoh} for all classical and some exceptional K\"ahler C-spaces with $b_2=1$. In other words, we always have $H\leq \nu$, so if $\nu < \mu$, then $Ric^{\perp}>0$.

By Table 12 in \cite{Itoh}, we know that for $(E_6, \alpha_1)=(E_6, \alpha_6)$ or $(E_7, \alpha_7)$, $\nu=2$, while $\mu = 12$ or $18$, so  $Ric^{\perp}>0$ for the two exceptional Hermitian symmetric spaces.

For $(B_r, \alpha_i)$ where $r\geq 3$ and $1\leq i\leq r$, by the computation in \S 3.1 in the paragraph right before Lemma 3.2 in \cite{ChauTam}, we see that $\mu = 2r-i \geq r \geq 3$, while by Table 3-5 of \cite{Itoh}, we see that $\nu =2$ or $1$, hence $ \nu < \mu$ thus $Ric^{\perp}>0$.

For $(D_r, \alpha_i)$ where $r\geq 4$ and $1\leq i\leq r$,  \S 3.2 of \cite{ChauTam} says that $\mu =2r-i-1\geq r-1\geq 3$, while Table 9-11 of \cite{Itoh} says that $\nu =2$, hence $\nu <\mu$ and $Ric^{\perp}>0$.

For $(C_r, \alpha_i)$ where $r\geq 3$ and $1\leq i\leq r$, \S 3.3 (the paragraph right before Lemma 3.5) of \cite{ChauTam} says that $\mu =2r-i+1\geq r+1\geq 4$, while Table 6-8 of \cite{Itoh} says that $\nu =2$, or $4$ when $i=r$, hence $\nu <\mu$ and $Ric^{\perp}>0$ when $i<r$ or when $r\geq 4$. In the case $r=3$ and $i=3$, namely for $M^6_{\mbox {\tiny{I\!I\!I} } } =(C_3, \alpha_3)$,  we have $\mu=4$ and $\nu=4$, so in order to conclude $Ric^{\perp}>0$ we need to show that the maximum of holomorphic sectional curvature of $(C_3, \alpha_3)$ is strictly less than $4$. To see that, we will just carry out the computation in its non-compact dual, namely, the bounded symmetric domain $D^{\mbox {\tiny{I\!I\!I} } }_3$ which is the set of all $3\times 3$ complex symmetric matrices $Z$ satisfy $I_3- \overline{Z}Z > 0$, as the curvature tensor just differs by a sign.

Recall that for type $I$ bounded symmetric domain $D^{\mbox {\tiny{I} } }_{\!p,q}$ which consists of all complex $p\times q$ matrices $Z$ such that $I_q- Z^{\ast }Z >0$. Let $\Phi (Z) = \log \det (I_q - Z^{\ast }Z)$. Then $\omega_g = \sqrt{-1} \partial \overline{\partial } \Phi (Z)$ is the standard metric on $D^{\mbox {\tiny{I} } }_{\!p,q}$. Write $Z=(z_{i\alpha })$, then at the origin $0$, the matrix of the metric is the identity matrix: $g_{i\alpha \, \overline{j\beta}}= \delta_{ij} \delta_{\alpha \beta}$,  and the curvature tensor is given by
$$ R_{i\alpha \overline{j\beta} k\gamma \overline{\ell \delta }} = - \delta_{ij}\delta_{k\ell} \delta_{\alpha \delta } \delta_{\gamma \beta} - \delta_{i\ell }\delta_{kj} \delta_{\alpha \beta } \delta_{\gamma \delta} .$$
So for tangent vectors $X=\sum_{i, \alpha } X_{i\alpha } \frac{\partial }{ \partial z_{i\alpha }}$ and $Y=\sum_{i, \alpha } Y_{i\alpha } \frac{\partial }{ \partial z_{i\alpha }}$, we have
$$ - R_{X\overline{X}Y\overline{Y}} = \sum_{i, \alpha , \beta} X_{i\alpha }\overline{Y}_{k\alpha } \,\overline{X}_{i\beta } Y_{k\beta } +\sum_{i, \alpha , \gamma} X_{i\alpha } \overline{Y}_{i\gamma  }\, \overline{X}_{k\alpha  } Y_{k\gamma }  = \rho (XY^{\ast }) + \rho (\, ^t\!X \overline{Y}),$$
where $\rho (A) = \sum_{i,j} |A_{ij}|^2$ for any matrix $A=(A_{ij})$. Since
$$ \rho (AA^{\ast }) = \sum_{i,j} |\sum_k A_{ik} \overline{A_{jk}}|^2 \leq \sum_{i,j} \rho_i(A)\rho_j(A) = \big(\sum_i \rho_i(A)\big)^2 = \rho(A))^2, $$
where $\rho_i(A) = \sum_k |A_{ik}|^2$, and similarly, $\rho (\, ^t\!A\overline{A}) \leq (\rho(A))^2$, we obtain that for any tangent vector $X\neq 0$ at the origin,  the holomorphic sectional curvature in the direction of $X$ satisfies:
$$ -H(X) = -R_{X\overline{X}X\overline{X}}/ |X|^4 = (\rho (XX^{\ast})+\rho (\, ^t\!X \overline{ X}))/ |X|^4 \leq 2 \rho(A))^2 / |X|^4 = 2.$$
Now the type $I\!I\!I$ bounded symmetric domain $D^{\mbox {\tiny{I\!I\!I} } }_r$ is a totally geodesic subspace in $D^{\mbox {\tiny{I} } }_{r,r}$, so its holomorphic sectional curvature in any tangent direction is greater than or equal to $-2$. While for $r=3$, it is easy to see from the above bisectional curvature formula that the Ricci curvature of $D^{\mbox {\tiny{I\!I\!I} } }_3$ is $-4$. So its compact dual $(C_3, \alpha_3)$ satisfies $\mu =4$ and $H\leq 2$, thus $Ric^{\perp}>0$. This completes the proof of Theorem 1.1.

 For Corollary 1.2, since a product metric will have $Ric^{\perp} >0$ if both of its factors are so and with nonnegative Ricci, we know a compact Hermitian symmetric space will have $Ric^{\perp}>0$ if it does not contain ${\mathbb P}^1$ as a factor. On the other hand, by Corollary 1.7, any ${\mathbb P}^1\times N$ cannot admit any K\"ahler metric with $Ric^{\perp}>0$. So Corollary 1.2 holds. We should remark that the main part of the computations here were done by \cite{Itoh} and \cite{ChauTam}, which led us to conclude that the condition $Ric^{\perp}>0$ is satisfied by all K\"ahler C-spaces with $b_2=1$ and $n\geq 2$ except the ones in ${\mathcal E}_0$:
$$ {\mathcal E}_0 = \{ (F_4, \alpha_3), \ (E_6, \alpha_4), \ (E_7,\alpha_i)_{i=3,4,6}, \ (E_8,\alpha_i)_{i=3,4,5,6,7}\} .$$
We believe that each space in ${\mathcal E}_0$ also satisfies $Ric^{\perp}>0$, but we can not claim that since we did not go through the lengthy computation here.

\vspace{0.5cm}

\section{Projectivized bundles}

In this section, we will consider projectivized bundles that admit $Ric^{\perp}>0$ metrics. Let $(M^n,g)$ be a compact K\"ahler manifold and $(E,h)$ be a holomorphic vector bundle over $M$ equipped with a Hermitian metric. Let $\pi : P={\mathbb P}(E^{\ast}) \rightarrow M$ be the projectivized bundle associated with $E$, namely, for any $x\in M$, the fiber $\pi^{-1}(x)={\mathbb P}(E_x)$ is the projective space of all complex lines in $E_x$ through the origin.

Let $L$ be the line bundle on $P$ dual to the tautological subbundle, determined by the short exact sequence
$$ 0 \rightarrow {\mathcal O}_P  \rightarrow \pi^{\ast}E^{\ast } \otimes L \rightarrow T_{P|M} \rightarrow 0 ,$$
where $T_{P|M} = \mbox{ker} (d\pi : T_P \rightarrow \pi^{\ast }T_M)$ is the relative tangent bundle. As is well known, the metric $h$ induces naturally a Hermitian metric $\hat{h}$ on $L$, whose curvature form is
\begin{equation}
C_1(L,\hat{h}) = \omega_{\mbox{\tiny{FS}}} - \frac{\sqrt{\!-\!1}}{|v|^2} \Theta^h_{v\overline{v}}
\end{equation}
at any point $(x,[v])\in P$, where $x\in M$ and $0\neq v\in E_x$. Here $\omega_{\mbox{\tiny{FS}}}$ is the K\"ahler form of the Fubini-Study metric on the fiber of $\pi$. Consider the closed $(1,1)$ form on $P$:
\begin{equation}
\omega_{G} = \lambda \pi^{\ast } \omega_g + C_1(L, \hat{h}),
\end{equation}
where $\lambda >0$ is a constant. Clearly, for $\lambda$ sufficiently large, $G$ is a K\"ahler metric on $P$.

Historically, the metric $G$ was used in \cite{Yau74} (Proposition 1) to show that, for any compact K\"ahler manifold $M^n$ and any holomorphic vector bundle $E$ of rank at least $2$ on $M$, the metric $G$ on the projectivized bundle $P$ has positive scalar curvature when $\lambda $ is sufficiently large. At about the same time, in \cite{Hitchin}  it was shown that, when $E={\mathcal O}\oplus {\mathcal O}(-k)$ on ${\mathbb P}^1$  where $k\geq 0$, so $P={\mathbb F}_k$ is the Hirzebruch surface, the metric $G$ has positive holomorphic sectional curvature when $\lambda $ is sufficiently large. In \cite{AHZ}, this later construction of  Hitchin was generalized to conclude that, when $(M,g)$ has positive holomorphic sectional curvature and $E$ is any Hermitian vector bundle over $M$, then for sufficiently large $\lambda$, the K\"ahler metric  $G$ on $P={\mathbb P}(E^{\ast})$ always has positive holomorphic sectional curvature.

In the following, we will follow the notations in \cite{AHZ} to compute the $Ric^{\perp}$ of $G$. Fix a point $(x,[v])\in P$, where $x\in M$, $v\in E_x$, and $|v|=1$. Choose local holomorphic coordinates $z=(z_1, \ldots , z_n)$ near $x$ in $M^n$, so that $x$ corresponds to $z=0$ and $g_{i\overline{j}}(0)=\delta_{ij}$, $dg(0)=0$. Also, choose a local holomorphic frame $e=(e_1, \ldots e_r)$ for $E$ near $x$, such that $e_1(0)=v$, and $h_{\alpha \overline{\beta}}(0)=\delta_{\alpha \beta}$, $dh(0)=0$. We can further assume that at $0$ we have $\partial_i\partial_jh=0$ and  $R^h_{v\overline{v}i\overline{j}}=\delta_{ij}\xi_i$ for any $1\leq i,j\leq n$ . Consider the section $u=e_1(z) + \sum_{\alpha =2}^r t_{\alpha}e_{\alpha}$ of $E$. Then $(z,t) = (z_1, \ldots , z_n, t_2, \ldots , t_r)$ forms a local holomorphic coordinate in $P$ near $(x,[v])$. For the sake of convenience, we will write $t_1=1$, and abbreviate $\frac{\partial }{\partial z_i}$ as $i$, and $\frac{\partial }{\partial t_{\alpha }}$ as $\alpha$, etc.. We then have
$$ h_{u\overline{u}} = \sum_{\alpha , \,\beta =1}^r t_{\alpha} \overline{t_{\beta}} h_{\alpha \overline{\beta}} (z), \ \ \ \ \ \ \omega_G = \lambda \, \pi^{\ast}\omega_g + \sqrt{\!-\!1}\partial \overline{\partial } \log h_{u \overline{u}}. $$
As in \cite{AHZ}, by a straight forward computation, we get at the origin that
\begin{eqnarray*}
&& G_{i\overline{j}} = \delta_{ij} (\lambda - \xi_i), \ \ \ \ G_{i\overline{\beta}} =0, \ \ \ \  G_{\alpha \overline{\beta}} = \delta_{\alpha \beta}; \\
& & G_{\ast \overline{\beta},\ast} = G_{\alpha \overline{j},\beta} =0, \ \ \ \ G_{i\overline{j},\alpha } = - R^h_{\alpha \overline{v}i\overline{j}}, \ \ \ \ G_{i\overline{j},k } = h_{u\overline{u},i\overline{j}k}; \\
&& G_{i\overline{\beta},k \,\overline{\delta }} = G_{ \alpha \overline{\beta}, \gamma \overline{j}} =0, \ \ \ \ G_{i\overline{j},\alpha \overline{\beta }} = -R^h_{\alpha \overline{\beta} i\overline{j}} + \delta_{\alpha \beta} \delta_{ij} \xi_i\, ;\\
&& G_{i\overline{j},k \overline{\beta }} = h_{u\overline{\beta}, i\overline{j}k}, \ \ \ \ G_{\alpha\overline{\beta},\gamma \,\overline{\delta }} = - h_{\alpha \overline{\beta}} h_{\gamma \overline{\delta}} - h_{\alpha \overline{\delta}} h_{\gamma \overline{\beta}}; \\
&& G_{i\overline{j},k \overline{\ell } } = \lambda \,g_{i\overline{j},k \overline{\ell } } + h_{u\overline{u}, i\overline{j}k \overline{\ell } } - h_{u\overline{u},i \overline{j } } h_{u\overline{u},k \overline{\ell } } - h_{u\overline{u},i \overline{\ell } } h_{u\overline{u},k \overline{j } }.
\end{eqnarray*}
Now let  $X=y+\sigma = \sum_{i=1}^n y_i \frac{\partial }{\partial \,z_i} + \sum_{\alpha =2}^r \sigma_{\alpha}\frac{\partial }{\partial \,t_{\alpha} }$ be any non-zero tangent vector of type $(1,0)$ at $(x,[v])\in P$. Denote by $R$ the curvature tensor of $G$. At the origin, we have
$$      R_{a\overline{b}c\overline{d}}  = - G_{a\overline{b}, c\overline{d} } +  \sum_{j=1}^n  \frac{1} {\lambda -\xi_j }   G_{a\overline{j},c} \overline{ G_{b\overline{j},d} } $$
for any indices $a$, $b$, $c$, $d$ which could be any $i$ or $\alpha$. Write $\eps_j= \frac{1} {\lambda -\xi_j }$. We have
\begin{eqnarray*}
&& R_{y\overline{\sigma}y\overline{\sigma}} = R_{y \overline{\sigma}\sigma \overline{\sigma}} =  0, \ \ \ \ R_{\sigma \overline{\sigma}\sigma \overline{\sigma}} = 2|\sigma |^4, \\
&& R_{y\overline{y}y\overline{\sigma}} = -h_{u\overline{\sigma},\, y\overline{y}y} - \sum \eps_j h_{u\overline{u},\,y\overline{y}y}R^h_{v\overline{\sigma} j\overline{y}} , \\
&& R_{y\overline{y}\sigma\overline{\sigma}} = R^h_{\sigma\overline{\sigma}y\overline{y}} - |\sigma |^2\sum_j \xi_j|y_j|^2 + \sum_j \eps_j |R^h_{v\overline{\sigma} j\overline{y}}|^2 , \\
&& R_{y\overline{y}y\overline{y}} = \lambda R^g_{y\overline{y}y\overline{y}} - h_{u\overline{u}, \, y\overline{y}y\overline{y}} + 2( \sum_j \xi_j |y_j|^2)^2 + \sum_j \eps_j |h_{u\overline{u}, \, y\overline{y}j }  |^2.
\end{eqnarray*}
Similarly, the Ricci curvature of $G$ at the origin is given by
\begin{eqnarray*}
 R_{a\overline{b}} & =  & \sum_c \frac{1}{ G_{c\overline{c}} } R_{a\overline{b}c\overline{c}} \ = \ \sum_j \eps_j R_{a\overline{b}j\overline{j}} +\sum_{\alpha }R_{a\overline{b}\alpha \overline{\alpha }} \\
 & = & - \sum_j \eps_j \,G_{j\overline{j},a \overline{b }} + \sum_{j,\,\ell } \eps_j\,\eps_{\ell} \,G_{a\overline{\ell},j } \overline{ G_{b\overline{\ell},j }}  - \sum_{\alpha } G_{a \overline{b }, \alpha \overline{\alpha }}.
\end{eqnarray*}
So we have
\begin{eqnarray*}
 R_{y\overline{\sigma }} & =  & - \sum_j \eps_j h_{u\overline{\sigma }, \, j\overline{j}y} - \sum_{j\, \ell } \eps_j \eps_{\ell } h_{ u\overline{u }, \, j \overline{\ell }y} R^h_{u\overline{\sigma }\ell \overline{j }} \\
 R_{\sigma \overline{\sigma }} & =  &  \sum_j \eps_j R^h_{\sigma \overline{\sigma }j\overline{j}} - |\sigma |^2 \sum_j \eps_j\xi_j  + \sum_{j,\, \ell } \eps_j \eps_{\ell } | R^h_{ \sigma \overline{v } j \overline{\ell }} |^2 + r |\sigma |^2 \\
 R_{y\overline{y }} & =  & \lambda \sum_j \eps_j R^g_{j\overline{j}y\overline{y}} - \sum_j \eps_j h_{u\overline{u}, \, y\overline{y}j\overline{j}}  + (\sum_j\eps_j\xi_j )\, (\sum_i\xi_i |y_i|^2) + \sum_j \eps_j |\xi_j|^2 |y_j|^2 +\\
 & & + \sum_{j, \, \ell} \eps_j \,\eps_{\ell } \, | h_{u\overline{u},\, y\overline{\ell}j} |^2  + \sum_{\alpha } R^h_{\alpha \overline{\alpha }y \overline{y}}  - (r-1) R^h_{v\overline{v}y\overline{y}}  + \sum_{\alpha , j} \eps_j | R^h_{\alpha \overline{v} y\overline{j}} |^2.
\end{eqnarray*}
Denote by $|y|^2=\sum_j |y_j|^2$, and $||y||^2=\sum_j (\lambda - \xi_j)|y_j|^2$. Consider the quantity
$$ \Phi = ||X||^2 R_{X\overline{X}} - R_{X\overline{X}X\overline{X}} = \Phi_0 + 2Re(\Phi_1) + \Phi_2 + 2Re(\Phi_3) + \Phi_4,$$
where
\begin{eqnarray*}
\Phi_0 & = & ||y ||^2 R_{y\overline{y}} - R_{y\overline{y}y\overline{y}}, \\
\Phi_1 & = & ||y||^2 R_{y\overline{\sigma}} - R_{y\overline{y}y\overline{\sigma }}, \\
\Phi_2 & = & ||y||^2 R_{\sigma \overline{\sigma }} + |\sigma |^2 R_{y\overline{y}} - 4R_{y\overline{y}\sigma \overline{\sigma }}, \\
\Phi_3 & = &  |\sigma |^2 R_{ y\overline{\sigma } }, \\
\Phi_4 & = & |\sigma |^2  R_{\sigma \overline{\sigma } } - R_{ \sigma \overline{\sigma } \sigma \overline{\sigma }}
\end{eqnarray*}
since $ R_{ y\overline{\sigma } y \overline{\sigma }}= R_{ y\overline{\sigma } \sigma \overline{\sigma }} = 0$. Let us write $\frac{1}{\lambda }=\eps$. When $\lambda $ is sufficiently large, we have
\begin{eqnarray*}
\Phi_0 & = & \lambda \left(  Ric^g_{\tilde{y}\overline{\tilde{y}}} - R^g_{\tilde{y}\overline{\tilde{y}}\tilde{y}\overline{\tilde{y}}} + \sum_{\alpha =2}^r R^h_{\alpha \overline{\alpha }\tilde{y}\overline{\tilde{y}}}  - (r-1) R^h_{v\overline{v}\tilde{y}\overline{\tilde{y}}} +O(\eps ) \right)  |y|^4,\\
\Phi_1 & = & |y|^3 |\sigma | O(1), \\
\Phi_2 & = & \big( \lambda r + O(1) \big) |y|^2 |\sigma |^2, \\
\Phi_3 & = & O(\eps ) |y| |\sigma |^3, \\
\Phi_4 & = &  \big( (r-2) + O(\eps ) \big) |\sigma |^4
\end{eqnarray*}
where $\tilde{y}=\frac{y}{|y|}$. Note that the condition (\ref{eq:curvature}) in Theorem 1.3 ensures that $\Phi_0>0$ when $\lambda $ is sufficiently large and $y\neq 0$. So under this condition and that $r\geq 3$, it is easy to see that the quantity $\Phi >0$ for any $0\neq X=y+\sigma$. Thus the metric $G$ has $Ric^{\perp }>0$ at the origin, hence everywhere on $P$. This completes the proof of Theorem 1.3. \qed

The verification of the curvature conditions in Example 1.4 is straight forward, so we omit it. We will postpone the proof of Theorem 1.5 to the next section, after we obtain some geometric consequences for the curvature condition $Ric^\perp >0$ first.

\vspace{0.5cm}

\section{Geometric consequences of $Ric^{\perp}>0$}

We begin with the proof of Theorem 1.6, which is a slight modification of an argument of T. Frankel \cite{Frankel}:

\begin{proof} Let $(M^n,g)$ be a compact K\"ahler manifold with $Ric^{\perp}>0$, and $Y_1$, $Y_2$ be two irreducible divisors  in $M^n$. We want to show that $Y_1$ and $Y_2$ always intersect each other. Assume the contrary, namely, assume that $Y_1\cap Y_2=\phi$, we want to derive a contradiction from that.

Let $\gamma : [0,\ell ] \rightarrow M^n$ be a unit speed geodesic from $Y_1$ to $Y_2$ which realizes the distance between them. Write $p=\gamma (0)\in Y_1$ and  $q=\gamma (\ell )\in Y_2$. Denote by $H_p\cong {\mathbb C}^{n-1}$ the $J$-invariant linear subspace of the tangent space $T_pM$ which is perpendicular to $\gamma'(0)$. It is just the orthogonal complement of $\mbox{sp} \{ \gamma'(0), J\gamma'(0)\}$. Define $H_q$ similarly.

 Since $M^n$ is K\"ahler, we have $\nabla J=0$ where $J$ is the almost complex structure of $M$, so the parallel translation along $\gamma $ will send $H_p$ onto $H_q$, as they are the complex hyperplanes of the tangent space of $M$ perpendicular to $\gamma '$.

For a unit vector field $X$ parallel along $\gamma$, the second variation of arc length is given by
$$ L''_X(0) = \langle \nabla_XX, \gamma' \rangle_q - \langle \nabla_XX, \gamma' \rangle_p - \int_0^{\ell } K(\gamma' \wedge X) dt. $$
To derive at a contradiction, we want to find such an $X$ with negative second variation. Let $\{ \eps_1, \ldots , \eps_{2n}\}$ be an orthonormal tangent frame of $M^n$ at $p$, such that $\eps_1=\gamma'$ and $\eps_{n+i}=J\eps_i$ for each $1\leq i\leq n$. Parallel translate them  along $\gamma$ and denote by the same letters. We see that both $H_p$ and $H_q$ are spanned by $\{ \eps_2, J\eps_2, \ldots , \eps_n, J\eps_n\}$.

For each $2\leq i\leq n$, we can find a complex curve $C$ through $p$ in a neighborhood of $p$ such that $T_pC$ is spanned by $\{ \eps_i, J\eps_i\}$. Extend $\eps_i$ to a vector field $X$ in $C$, then we have
$$ \nabla_XX + \nabla_{JX}JX = J( -\nabla_XJX + \nabla_{JX}X) = J[JX, X]. $$
Since $[JX, X]$ is in $C$, at $p$ it is perpendicular to $\gamma'$, hence we have
$$ \langle \nabla_{\eps_i} \eps_i + \nabla_{J\eps_i} J\eps_i , \gamma' \rangle_p = 0, $$
and similarly the same equality holds at $q$, so we get
$$ \sum_{i=2}^n ( L_{\eps_i}'' + L_{J\eps_i}'' ) = -\sum_{i=2}^n \int_0^{\ell } K(\gamma'\wedge \eps_i) + K(\gamma'\wedge J\eps_i) = - \int_0^{\ell } Ric^{\perp} (\gamma') < 0.$$
So at least one of the terms in the left hand side will be negative. Now if both $Y_1$ and $Y_2$ are smooth, then we must have $T_pY_1$ and $T_qY_2$ perpendicular to $\gamma'$, hence $H_p=T_pY_1$ and $H_q=T_qY_2$. So the above negative second variation term will contradict the  fact that $\gamma$ is the shortest geodesic from $Y_1$ to $Y_2$.

If $p$ is a singular point of $Y_1$, let us denote by $C_p\subset T_pM$ the tangent cone of $Y_1$ at $p$. It is a subvariety in the tangent space. We claim that the support (reduced part) of $C_p$, which we will still denote by the same letter for convenience,  coincides with $H_p$.  It suffices to show that $C_p\subset V_p$, where $V_p\cong {\mathbb R}^{2n-1}$ is the orthogonal complement of $\gamma '(0)$ in $T_pM$. Since $-v\in C_p$ for any $v\in C_p$, so if $C_p$ is not contained in $V_p$, then there will be $0\neq v\in C_p$ which makes an acute  angle with $\gamma'(0)$. By Theorem 11.8 of \cite{Whitney}, we know that for any $v\in C_p$, there exists a smooth arc $\sigma : [0,\epsilon ) \rightarrow Y_1$ such that $\sigma'(0)=v$. This will violate the assumption that $\gamma$ is the shortest curve between $Y_1$ and $Y_2$. So we have $C_p=H_p$, and similarly, $C_q=H_q$. Then the term with negative second variation along $\gamma$ will again contradict the fact that $\gamma$ is the shortest curve between $Y_1$ and $Y_2$. This completes the proof of Theorem 1.6.
\end{proof}

An equally effective approach is to work with the energy of a path $\gamma$, $\mathcal{E} (\gamma)$ (see for example \cite{Schoen-Wolfson} and \cite{NZ}).

The argument can be adapted to prove a Lefschetz type result for a pair of complex hypersurfaces $(Y_1, Y_2)$, or a hypersurface $Y$ in $M$. The key is that for any pair of hypersurfaces $Y_1, Y_2$, one may define $\Omega$ to be the space all paths originating from $Y_1$ and ending with $Y_2$. The energy of the path $\gamma\in \Omega$, $\mathcal{E}(\gamma)$ is defined as usual. It is well known that the critical points of the energy functional are geodesics which intersects $Y_i$ orthogonally (namely normal geodesics). The same argument as above implies the following index estimate, which includes the intersecting result as a consequence since the minimizers can be identified with  $Y_1\cap Y_2$ (cf. \cite{Schoen-Wolfson}).

\begin{corollary}\label{coro-41} Let $\gamma$ be a nontrivial critical point (namely a nonconstant normal geodesic after \cite{NZ}). Then the index of $ind(\gamma)\ge 1$. In particular,
\begin{equation}\label{eq:41}
\pi_0(\Omega, Y_1\cap Y_2)=\{0\},\quad \iota_*: \pi_1(Y_1, Y_1\cap Y_2)\to \pi_1(M, Y_2) \mbox{ is surjective}. \end{equation}
When $Y_1=Y_2=Y$, this implies that $\pi_1(M, Y)=\{0\}$.
\end{corollary}
\begin{proof} The index estimate follows verbatim from the above argument in proving that $Y_1\cap Y_2\ne \emptyset$. For rest claims the argument of \cite{Schoen-Wolfson} via the Morse theory and exact sequences  applies (cf. \cite{Ni-Wolfson}).
\end{proof}
Note that in \cite{NZ} it was  conjectured that $\pi_1(M)=\{0\}$. The last statement of the corollary is clearly a consequence of an affirmative answer to the conjecture.

Next we prove  the following geometric property for manifolds with $Ric^{\perp}>0$, which will be a key factor in determining the low dimensional cases:

\begin{theorem}\label{thm:41}
Let $(M^n,g)$ be a compact K\"ahler manifold with $Ric^{\perp}>0$. Let $C$ be an irreducible curve in $M$, with $f: \tilde{C} \rightarrow M$ the normalization of $C$. Denote by $g$ the genus of $\tilde{C}$ and $K_{\!M}$ the canonical line bundle of $M$. Then we have $K_M^{\!-\!1}C \geq 3-2g$.  In particular, $K_M^{\!-\!1}C \geq 3$ for any rational curve $C\subseteq M$.
\end{theorem}

\begin{proof}
First, note that  the holomorphic sectional curvature of $(M^n,g)$ is a scalar-valued function $H$ on the projectivized cotangent bundle $\pi: {\mathbb P}(\Omega_M) \rightarrow M$:
$$ H([X]) = R_{X\overline{X}X\overline{X}}/|X|^4, $$
where $X$ is any nonzero type $(1,0)$ tangent vector in $M$. If $U$ is a piece of smooth complex curve in $M$, then the inclusion map $i: U\rightarrow M$ has a lift $\tilde{i}: U \rightarrow {\mathbb P}(\Omega_M)$. For any $x\in U$, consider
$$ Ric^{\perp}|_U= Ric^{\perp}_{X\overline{X}} /|X|^2, $$
where $X$ is any non-zero type $(1,0)$ tangent vector of $U$ at $x$.  It is a well defined function on $U$, and we have
$$ Ric^\perp|_U\ i^{\ast} \!\omega_g = Ric_g|_U - \tilde{i}^{\ast }\!H \ i^{\ast}\! \omega_g \geq Ric_g|_U - \Theta( i^{\ast}\! \omega_g),$$
where $Ric_g $ is the the Ricci $(1,1)$ form of $\omega_g$, and $\Theta (i^{\ast}\! \omega_g)$ is the curvature $(1,1)$ form of the restriction metric $i^{\ast}\! \omega_g=\omega_g|_U$. The inequality is due the decreasing property of curvature for complex submanifolds.

Now suppose that $C$ is an irreducible complex curve in $M$, and denote by $U$ its smooth part. Let $f: \tilde{C} \rightarrow C\subset M$ be the normalization of $C$, and write $\tilde{U}=f^{-1}(U)\subset \tilde{C}$. Since $Ric^\perp >0$ on $M$, by integrating the positive function $Ric^{\perp}|_U$ over $U$, we get
\begin{eqnarray*}
 K_M^{-1}C & = & \int_C Ric_g \ \  = \ \int_U Ric_g \\
 &  > & \int_U \Theta (i^{\ast }\!\omega_g) \ = \ \int_{\tilde{U}} \Theta (f^{\ast}\!\omega_g) \\
 & = & \int_{\tilde{C}\setminus \tilde{D}} \Theta (f^{\ast}\!\omega_g),
\end{eqnarray*}
where $D$ is the divisor in $\tilde{C}$ given by the  zeroes of $df$, and $\tilde{D}$ is the support of $D$.
Note that on $\tilde{C}$, $f^{\ast}\!\omega_g$ is a degenerate metric, with zeroes at $D$. In fact, $f^{\ast}\!\omega_g$ is a Hermitian metric on the holomorphic line bundle $T_{\tilde{C}}(D)$, with
$$ \int_{\tilde{C}} \Theta (f^{\ast }\!\omega_g) = c_1 (T_{\tilde{C}}(D)) = 2-2g + \deg (D), $$
and the integral of $\Theta (f^{\ast}\!\omega_g)$ over $\tilde{C}\setminus \tilde{D}$ is just $2-2g$. Since all terms involved are integers, the strict inequalities above gives
$ K_M^{-1}C \geq 3-2g$, thus completing the proof of the theorem.
\end{proof}

For a smooth rational curve $C\subset M$, we have the short exact sequence of vector bundles over $C$
$$ 0 \rightarrow T_C \rightarrow T_M|_C \rightarrow  N_C \rightarrow  0 ,$$
where $N_C$ is the normal bundle of $C$ in $M$. By taking their first Chern classes, we get
$$ c_1(N_C) = c_1(T_M|_C)- c_1(T_C) = K^{\!-\!1}_MC - 2 >0.$$
In other words, we have the following:

\begin{corollary}
For any smooth rational curve $C$ in a compact K\"ahler manifold $M^n$ with $Ric^{\perp}>0$, the normal bundle of $C$ has positive first Chern class.
\end{corollary}

As an immediate consequence, we know that if $M^n$ is the product ${\mathbb P}^1 \times N$, or more generally, if there is a morphism $f: M \rightarrow N$ where a generic fiber is a smooth rational curve, then $M^n$ cannot admit any K\"ahler metric with $Ric^{\perp}>0$. In particular, the K\"ahler C-space ${\mathbb P}(T_{\!{\mathbb P} ^{\,2}})$ does not admit such a metric since it is a ${\mathbb P}^1$-bundle.

Next let us prove Theorem \ref{thm:15}. We will divide the proof into three steps. In the first step, we prove that the base surface must be ${\mathbb P}^2$. In the second step, we show that the fiber bundle must be the projectivization of a vector bundle. Finally, in step three, we show that the vector bundle must be the trivial bundle tensoring with a line bundle.

\begin{lemma}
Let $p: P^n \rightarrow S$ be a holomorphic fiber bundle over a compact complex surface $S$, with fiber ${\mathbb P}^{n-2}$ where $n\geq 4$. Assume that $P^n$ admits a K\"ahler metric with  $Ric^\perp>0$. Then $S$ must be biholomorphic to ${\mathbb P}^2$.
\end{lemma}

\begin{proof} First of all, $S$ is projective since $P^n$ is so. If $C$ is a $(-1)$ or $(-2)$ curve in $S$, namely, a smooth rational curve with self intersection number $-1$ or $-2$, then since $C$ can be blown down to a smooth or singular point, by considering the proper transform of a smooth curve down stair away from that point, we know that there exists a smooth curve $D$  in $S$ which does not intersect $C$. Then the smooth hypersurfaces $p^{-1}(C)$ and $p^{-1}(D)$ in $P$ do not intersect, violating Theorem 1.6. So $S$ cannot contain any $(-1)$ or $(-2)$ curve, hence is a minimal surface.

By Corollary 1.7, we know that $P$, hence $S$,  cannot fiber over a curve. Let $\kappa $ be the Kodaira dimension of the minimal algebraic surface $S$. If $\kappa = -\infty $, then the only choice for $S$ is ${\mathbb P}^2$ since it cannot be ruled. If $\kappa =0$, then a finite cover of $S$ is either a complex torus or a K3 surface, which admits a non-trivial holomorphic $2$-form. Pulling it back to the finite cover of $P$, we get a violation to the vanishing theorem in \cite{NZ} (Theorem 1.7). When $\kappa =1$, $S$ is an elliptic surface, which is not possible. So we are left with the case of $\kappa =2$, namely, $S$ is a general type surface.

Since $S$ does not contain any $(-2)$ curve, its canonical line bundle $K_{\!S}$ is ample. From the fact that $P$ has finite fundamental group and does not have any holomorphic $1$ or $2$-form on it, we know that $S$ satisfies $q=p_g=0$, thus
$$ \chi({\mathcal O}_S) = 1-q+p_g= \frac{1}{12}(c_1^2+c_2) = 1,$$
and $S$ must be simply-connected. So $S$ is homeomorphic to (but not diffeomorphic to) ${\mathbb P}^2 \# k\overline{{\mathbb P}^2}$, the blowing up of ${\mathbb P}^2$ at $k$ general points, with $1\leq k\leq 8$.

By Riemann-Roch Theorem, we get $h^0(2K_{\!S})=1+c_1^2=10-k\geq 2$. Take a non-trivial global holomorphic section $\sigma$ of the line bundle  $K_{\!S}^{\otimes 2}$ on $S$. Then $s=p^{\ast}\sigma$ is a non-trivial global holomorphic section of $L^{\otimes 2}$ on $P$, where $L = p^{\ast }K_{\!S}$ is a sub line bundle of $\bigwedge^{\!2}\!\Omega_{\!P}$.

The K\"ahler metric on $P$ naturally induces metrics on $L$ and $L^{\otimes 2}$. By applying the Bochner formula for $|s|^2$, we know that at the point $x\in P$ where $|s|^2$ reaches its maximum, we have $\Theta_{L^{\otimes 2}} (s, \overline{s}, \cdot , \overline{\cdot }) \geq 0$, where $\Theta$ is the curvature form.  In a small neighborhood of $p(x)$, we may take a local holomorphic section $\tau $ of $K_{\!S}$, such that $\tau^2 =\sigma$. Then $t=p^{\ast }\tau$ is a local holomorphic $2$-form in $P$ such that $t^2=s$. We have at $x$ that
$$  \Theta_{L^{\otimes 2}} (s, \overline{s}, \cdot , \overline{\cdot })  = 2\Theta_{L} (t, \overline{t}, \cdot , \overline{\cdot }) \geq 0.$$
Note that $2$-forms correspond to skew-symmetric matrices under unitary frame of $P^n$, which can be diagonalized into $2\times 2$ blocks. Since we have $t\wedge t=0$, we know that there exists unitary frame $\{ e_1, \ldots , e_n\}$ and dual coframe $\{ \varphi_1, \ldots , \varphi_n\}$ of $P^n$ at $x$, such that $t=\lambda \varphi_1\wedge \varphi_2$ with $\lambda \neq 0$. So the above curvature condition becomes
$$ R_{1\overline{1}v\overline{v}}+ R_{2\overline{2}v\overline{v}} \leq 0 $$
for any type $(1,0)$ tangent vector $v$ at $x$. As in the proof of Theorem 1.7 in \cite{NZ}, we know this will lead to a contradiction to the condition $Ric^{\perp}>0$ on $P$. So the $\kappa =2$ case is not possible, and we have completed the proof of the lemma.
\end{proof}

The following fact should be well-known in algebraic geometry, and we learned it from Joe Harris many years ago.

\begin{lemma}
Let $X$ be a projective manifold with $h^2(X, {\mathcal O}^{\ast }) =0$, then any holomorphic ${\mathbb P}^{k\!-\!1}$-bundle over $X$ is the projectiviation of some rank $k$ holomorphic vector bundle over $X$.
\end{lemma}

In particular, for $X={\mathbb P}^m$ with any $m$, since $h^2({\mathcal O})= h^3({\mathbb Z})=0$, the exponential sequence gives $h^2( {\mathcal O}^{\ast }) =0$, hence the lemma applies.

Denote by ${\mathcal G}{\mathcal L}_{k}$ and ${\mathcal P}{\mathcal G}{\mathcal L}_{k}$  the (non-abelian) sheaf of holomorphic maps from $X$ into $GL_k({\mathbb C})$ or $PGL_k({\mathbb C})$, respectively. We have the short exact sequence of sheaves on $X$:
$$ 0 \rightarrow {\mathcal O}^{\ast } \rightarrow {\mathcal G}{\mathcal L}_{k} \rightarrow  {\mathcal P}{\mathcal G}{\mathcal L}_{k} \rightarrow 0.$$
From the vanishing of $H^2({\mathcal O}^{\ast})$, we get the surjection
$$ H^1(X, {\mathcal G}{\mathcal L}_{k}) \rightarrow H^1(X, {\mathcal P}{\mathcal G}{\mathcal L}_{k}) \rightarrow 0.$$
On the other hand, the isomorphism classes of holomorphic ${\mathbb P}^{k-1}$-bundles over $X$ are in one one correspondence with the cohomology classes in $H^1(X, {\mathcal P}{\mathcal G}{\mathcal L}_{k})$, while the isomorphism classes of holomorphic rank $k$ vector bundles are corresponding to elements of $H^1(X, {\mathcal G}{\mathcal L}_{k})$, so the statement of the lemma holds.

Combine Lemma 4.1 with Lemma 4.2, we know that the manifold $P$ in Theorem 1.5 must be in the form ${\mathbb P}(E)$ for some rank $r\geq 3$ holomorphic vector bundle $E$ over ${\mathbb P}^2$. We want to show that $P={\mathbb P}^2 \times {\mathbb P}^{r-1}$, or equivalently, $E\cong {\mathcal O}(k)\otimes {\mathcal O}^{\oplus r}$ for some integer $k$. We first prove the following

\begin{lemma}
Let $E$ be a rank $r\geq 3$ holomorphic vector bundle over ${\mathbb P}^2$, and $P={\mathbb P}(E)$ admits a K\"ahler metric $Ric^{\perp }>0$. Then for any line $L \subset {\mathbb P}^2$, $E|_L $ is the tensor product of a line bundle with the trivial bundle.
\end{lemma}

\begin{proof}
By Grothendieck Theorem, $E|_L= {\mathcal O}(a_1) \oplus \cdots \oplus {\mathcal O}(a_r)$ on $L\cong {\mathbb P}^1$.  Denote by $p: P\rightarrow {\mathbb P}^2$ the projection, and let $X=p^{-1}(L)= {\mathbb P}(E|_L)$.  It is not hard  to see that, for the section $C_i\subset X$ corresponding to the quotient line bundle ${\mathcal O}(a_i)$ of $E|_L$, its normal bundle $N_i$ in $X$ will have first Chern class
$$ c_1(N_i) = ra_i - (a_1 +\cdots + a_r).$$
So the normal bundle of $C_i$ in $P$ will have first Chern class equal to above number plus $1$, which has to be positive by Corollary 4.3. Thus $c_1(N_i)\geq 0$ for each $i$, implying that all $a_i$ are equal. This completes the proof of the lemma.
\end{proof}

The above lemma says that when $P={\mathbb P}(E)$ over ${\mathbb P}^2$ admits a K\"ahler metric with $Ric^{\perp}>0$, $E|_L={\mathcal O}(a)^{\oplus r}$ for any line $L \subset {\mathbb P}^2$. Since $\det(E|_L) = (\det E)|_L$, we see that $a$ is independent of $L$, and replacing $E$ by $E(-a)$, we may assume that $E|_L$ is trivial for any line $L$ in ${\mathbb P}^2$. In other words, the bundle $E$ on ${\mathbb P}^2$ is uniformly trivial. By Theorem 3.2.1 in \cite{Okonek}\footnote{We learn about this crucial fact from Jun Li, who kindly supplied us with a direct proof of this theorem in our special case.}, we know that $E$ itself must be trivial. This completes the proof of Theorem 1.5.

\vspace{0.5cm}

\section{Proof of Theorems \ref{thm:18} and \ref{thm:19}}

Let $(M^3,g)$ be a compact K\"ahler manifold with $Ric^{\perp}>0$. Then $M$ is projective and the scalar curvature is everywhere positive, thus $M$ is uniruled and $K_{\!M}$ is not nef. By the cone-contraction theorem of Mori \cite{Mori} and Kollar  \cite{Kollar}, the contraction map $\phi_R: M \rightarrow Z$ of an extremal ray $R\subset \overline{NE}(M)$ could only be one of the following:
\begin{eqnarray*}
(E) \ && \dim Z =3, \,\phi_R \ \mbox{is} \mbox{ birational, and}  \\
 (E_1) && Z \ \mbox{is smooth, and} \ \phi_R \ \mbox{is the inverse of the blowing up of a smooth curve in} \ Z,\\
 (E_2) && \phi_R \ \mbox{is the inverse of the blowing up of a smooth or singular (3 types) point in} \ Z; \\
 (C) \ && \dim Z =2, \  \phi_R \ \mbox{is a fibration over} \  Z \ \mbox{whose fibers are plane conics, and the generic } \\
 && \mbox{fibers are  smooth;}\\
 (D) \ && \dim Z =1, \  \mbox{the generic fibers of} \ \phi_R \ \mbox{are Del Pezzo surfaces;} \\
 (F) \ && \dim Z =0, \ M \ \mbox{is Fano.}
\end{eqnarray*}
Note that $(E1)$ is not possible as any ruling would be smooth rational curve whose normal bundle has negative first Chern class. $(E2)$ is also not possible as the exceptional divisor is smooth in each of the four cases, and one can take a smooth hypersurface in $Z$ avoiding the point of blown up, and then its pull back in $M$ would be another smooth hypersurface not intersecting the exceptional divisor, violating Theorem 1.6. Similarly, $(C)$ is not possible as a generic fiber would be a smooth rational curve with trivial normal bundle, and $(D)$ is not possible by Corollary 1.7, so we are only left with the Fano case.

In the Fano case, since we know by Theorem 4.2 that the pseudo index of $M$ must be at least $3$, so by the recent result of Dedieu and H\"oring \cite{DH}, which characterizes projective spaces and quadrics amongst all Fano manifolds by the condition $i(M)\geq \dim (M)$,  we know that $M^3$ must be either ${\mathbb P}^3$ or ${\mathbb Q}^3$, thus completing the proof of Theorem 1.8.

Alternatively, since Fano threefolds are fully classified, we could also derive at the conclusion of Theorem 1.8 without using the deep theorem of \cite{DH}. First to rule out the Picard number $\rho (M) >1$ case. This can be done either by the $n=3$ case of the generalized Mukai conjecture $\rho(M) (i(M)-1) \leq n$, which forces $\rho (M)=1$, or by recalling the results of Mori and Mukai \cite{MoriMukai} on the classification of all Fano threefolds with  $\rho (M) >1$, which have $88$ deformation families. Such a manifold is either {\em imprimitive,} meaning that $M^3$ is the blowing up of another Fano threefold $Z$ along a smooth curve, or {\em primitive,} which means otherwise. The imprimitive cases cannot occur as any ruling in the exceptional divisor would be a smooth rational curve whose normal bundle has negative first Chern class. In the primitive case, Mori and Mukai showed that (Theorem 5 of \cite{MoriMukai}) either $\rho (M)=2$ and $M$ is a conic fibration over ${\mathbb P}^2$, or $\rho (M)=3$ and $M$ is a conic fibration over ${\mathbb P}^1\!\times {\mathbb P}^1$. Neither could occur as the generic fiber would be  a smooth rational curve with trivial normal bundle.

For Fano threefolds with $\rho (M) =1$, which are called {\em prime} Fano threefolds, there are $17$ deformation families,  fully classified by Iskovskikh \cite{Iskov}. Let $r$ be the largest integer where $K_M^{\!-\!1}=rA$ for some ample divisor $A$ in $M^3$. $r$ is called the {\em index} of $M$. It is well known that $r=4$ if and only if $M\cong {\mathbb P}^3$, and $r=3$ if and only if $M\cong {\mathbb Q}^3$, the smooth quadric in ${\mathbb P}^4$. When $r=1$ or $r=2$, it is known that $M$ contains a line, namely, smooth rational curve $C$ with $C\cdot A=1$. See for instance \cite{IP} (Theorem 4.5.8). So the pseudo index is $1$ or $2$, contradicting Theorem 4.2. This completes the proof of Theorem 1.8.

Next let us focus on the $4$-dimensional case. Let $(M^4,g)$ be a compact K\"ahler manifold of dimension $4$ with $Ric^{\perp}>0$ everywhere. Then we know that $M^4$ is projective, simply-connected, and uniruled. Denote by $i(M)$ the pseudo index of $M^4$, namely, the minimum of $K^{-1}_{\!M}C $ for all rational curve $C$ in $M$. We know that $i(M)\geq 3$ by Theorem 4.2. Since $M^4$ cannot be the blowing up of a point by Corollary 1.7, so Theorem 1.1 of \cite{AO} implies that there is no non-nef extremal ray. In other wards, any extremal ray $R$ of $M$ must be nef, meaning that the associated contraction map $\phi_R$ is of fiber type. The target space cannot be of dimension one or three, by Corollary 1.7 or Corollary 4.3, respectively. So the target has to be of dimension two or zero.

By Part 6 of Theorem 4.1.3 in \cite{AM}, we know that either $M^4$ is a Fano fourfold with Picard number $\rho (M)=1$,  or $\phi_R$ is an equidimensional fibration over a normal surface with general fiber being a del Pezzo surface. In the latter case, by Theorem 1.3 of \cite{HN}, we know that $\phi_R$ is actually a projective bundle, and the target space is smooth. Now Theorem 1.5 kicks in and enables us to conclude that $M^4$ must be ${\mathbb P}^2\times {\mathbb P}^2$. This completes the proof of Theorem 1.9.

\vspace{0.5cm}

\section{Appnedix: K\"ahler manifolds with $Ric^\perp \equiv 0$}

We have seen that $Ric^\perp \equiv 0$ for any complex curve $M^1$, and it is natural to wonder what kind of K\"ahler manifolds in higher dimensions will have flat orthogonal Ricci curvature. On such a manifold $M^n$, where $n\geq 2$, we have
$$ |X|^2Ric_{X\overline{X}} = R_{X\overline{X}X\overline{X}} $$
for any type $(1,0)$ tangent vector $X$. By the symmetry properties of the curvature tensor, one can rewrite the above as
\begin{equation} R_{i\overline{j}k\overline{\ell }} = \frac{1}{4} \big(  R_{i\overline{j}} g_{k\overline{\ell }}  + R_{k\overline{\ell }} g_{i\overline{j}} + R_{i\overline{\ell }} g_{k\overline{j}} + R_{k\overline{j}} g_{i\overline{\ell }}  \big) , \label{eq:appendix}
\end{equation}
where $\{ e_1, \ldots , e_n\}$ is a local tangent frame of $M^n$ and $R_{i\overline{j}}$ are the components of the Ricci tensor. If we  choose a unitary frame $e$ such that $R_{i\overline{j}} = r_i \delta_{ij}$, then under this frame we have
$$ R_{i\overline{j}k\overline{\ell }} = \frac{1}{4} (r_i + r_k) \big(  \delta_{ij} \delta_{k\ell }+ \delta_{i\ell} \delta_{kj }  \big) .$$
By letting $i=j$ and $k=\ell$, we get
$$ R_{i\overline{i}i\overline{i}} = r_i,  \ \ \ R_{i\overline{i}k\overline{k}} = \frac{1}{4}(r_i+r_k) \ \ \mbox{if} \ i\neq k. $$
Fix $i$ and sum up over $k$, we get $r_i = r_i + \frac{1}{4}(n-2)r_i +S$, where $S$ is the scalar curvature. We have $S=0$ since $Ric^\perp =0$, so we get $(n-2)r_i=0$ for each $i$, namely, when $n>2$, the Ricci tensor, hence the curvature tensor, will be identically zero.

For $n=2$, since $S=0$, the equation (\ref{eq:appendix}) says exactly that the Weyl curvature tensor vanishes, hence $M^2$ is conformally flat. By \cite{Tanno}, we know that either $M^2$ is flat, or it is locally holomorphically isometric to a product of complex curves $C_1\times C_2$, where $C_1$ has constant curvature $a>0$ and $C_2$ has constant curvature $-a$. To summarize, we have the following:

\begin{theorem}
Let $M^n$ be a K\"ahler manifold $M^n$ with $Ric^\perp \equiv 0$. If $n\geq 3$, then $M^n$ is flat. If $n=2$, then $M^2$ is conformally flat, which means locally it is either flat or the product of two complex curves, with constant curvature of opposite values.
\end{theorem}

As an immediate consequence, one can state the following:

\begin{corollary}
Let $M^n$ be a compact K\"ahler manifold $M^n$ with $Ric^\perp \equiv 0$ and $n\geq 2$. Then either it is a finite under cover of a flat complex torus, or $n=2$ and $M^2={\mathbb P}(E)$ where $E$ is a unitary flat holomorphic vector bundle of rank two over a compact complex curve $\Sigma_g$ of genus $g\geq 2$.
\end{corollary}

There is also an alternative way to prove the above theorem, in which we view $Ric^\perp$ as the  {\it holomorphic sectional curvature} of an algebraic curvature operator risen from the one acting on the two-forms via the Bochner formula. Recall the notations from the appendix of \cite{NiTam} and define an algebraic (K\"ahler) curvature operator
$$
R_{Ric}=Ric \,\bar{\wedge} \operatorname{id},
$$
where for any $A, B:T'_xM \to T'_xM$ Hermitian symmetric  ($\bar{A}(X)=A(\bar{X})=0$),
\begin{eqnarray*}
 \langle A\bar{\wedge}B(X\wedge\bar{Y}), \overline{Z\wedge\bar{W}}\rangle &\doteqdot& \frac{1}{2} \left(	\langle \left(A\wedge \bar{B}+B\wedge\bar{A}\right)(X\wedge\bar{Y}), \overline{Z\wedge\bar{W}}\rangle\right.\\
 &\quad&\quad  \left.+\langle \left(A\wedge \bar{B}+B\wedge\bar{A}\right)(W\wedge\bar{Y}), \overline{Z\wedge\bar{X}}\rangle\right).
 \end{eqnarray*}
It is easy to check that $Ric^\perp (X, \overline{X})=H_{R_{Ric}-R}(X)/|X|^2$. Here $H_{\widetilde{R}}(X)$ is the holomorphic sectional curvature of $\widetilde{R}=R_{Ric}-R$.
From this it is easy to see that $Ric^\perp\equiv 0$ implies that $\widetilde{R}\equiv 0$. Hence $Ric^\perp \equiv 0$, via the decomposition of the curvature operators, induces that either $n=1$, or $n=2$ and $R$ is conformally flat, or $n\ge 3$ and $R$ is flat.

\vspace{0.5cm}

\begin{center}
\textbf{Acknowledgements}
\end{center}
We would like to thank Jun Li for pointing out and give outline of proof of the fact that any uniformly trivial vector bundle on projective space is trivial, which is crucial to the proof of Theorem 1.5,  and to thank Hsian-Hua Tseng for informing us about the reference \cite{DH}, which gives a much shorter proof of Theorem 1.8.

\end{document}